\documentclass[11pt,a4paper]{article}

\usepackage[english]{babel}
\usepackage[utf8]{inputenc}
\usepackage{lmodern}
\usepackage[T1]{fontenc}
\usepackage{datetime}
\usepackage{xifthen}

\usepackage{amsmath,amssymb}
\usepackage{mathtools}
\usepackage{geometry}
\usepackage{hyperref}
\usepackage{microtype}
\usepackage{verbatim}
\usepackage{rotating}
\usepackage{array}
\usepackage[ruled,linesnumbered]{algorithm2e}
\usepackage{tikz}
\usepackage{pgfplots}

\usepackage{xspace}
\newcommand{\ie}{i.e.\@\xspace}
\newcommand{\eg}{e.g.\@\xspace}
\newcommand{\etal}{et al.\@\xspace}

\usepackage{amsthm}
\newtheorem{theorem}{Theorem}
\newtheorem{definition}[theorem]{Definition}
\newtheorem{proposition}[theorem]{Proposition}

\newtheorem{corollary}[theorem]{Corollary}
\newtheorem{lemma}[theorem]{Lemma}
\newtheorem{conjecture}[theorem]{Conjecture}

\newcommand{\GF}[2][2]{\mathbb{F}_{#1^{#2}}}
\newcommand{\T}{\mathcal{T}}
\newcommand{\set}[1]{\left\{ #1 \right\}}

\DeclareMathOperator{\Tr}{Tr}
\makeatletter
\newcommand{\tr}[3][1]{\ifthenelse{\isempty{#3}}%
  {\Tr_{#1}^{#2}}%
  {\Tr_{#1}^{#2}\left(#3\right)}}
\newcommand{\addch}[1]{\ifthenelse{\isempty{#1}}%
  {\chi}%
  {\chi \left( #1 \right)}}
\newcommand{\mulch}[2][m_1]{\ifthenelse{\isempty{#2}}%
  {\psi_{#1}}%
  {\psi_{#1} \left( #2 \right)}}
\makeatother
\newcommand{\Wa}[1]{\widehat{\chi_{#1}}}
\newcommand{\Snu}[1][\nu]{S_{#1}(a, b, \omega)}
\newcommand{\mystery}{h_{a, b}(\omega)}

\allowdisplaybreaks

\hypersetup{pdftitle={A conjecture about Gauss sums and bentness of binomial Boolean functions},%
  pdfauthor={Jean-Pierre Flori},%
  pdfsubject={Mathematics and Cryptography},%
  pdfcreator={Jean-Pierre Flori},%
  pdfproducer={Jean-Pierre Flori},%
  pdfkeywords={mathematics} {cryptography} {Boolean functions} {bent functions} {Walsh spectrum} {Kloosterman sums},
}

\title{A conjecture about Gauss sums and bentness of binomial Boolean functions}

\author{Jean-Pierre Flori
  \thanks{ANSSI (Agence nationale de la sécurité des systèmes d'information),
    51, boulevard de La Tour-Maubourg,
    75700 Paris 07 SP, France.
    \texttt{jean-pierre.flori@ssi.gouv.fr}}
}

\date{}

\begin{document}

\maketitle

\begin{abstract}
  In this note, the polar decomposition of binary fields of even extension degree is
  used to reduce the evaluation of the Walsh transform of binomial Boolean functions
  to that of Gauss sums.
  In the case of extensions of degree four times an odd number, an explicit formula involving a Kloosterman sum is conjectured,
  proved with further restrictions, and supported by extensive experimental
  data in the general case.
  In particular, the validity of this formula is shown to be equivalent
  to a simple and efficient characterization for bentness
  previously conjectured by Mesnager.
\end{abstract}


\section{Introduction}
\label{sec:introduction}

Bent functions are Boolean functions defined over an extension of even
degree and achieving optimal non-linearity.
They are of both combinatorial and cryptographic interest.
Unfortunately, characterizing bentness of an arbitrary Boolean function
is a difficult problem,
and even the less general question of providing simple and efficient
criteria within infinite families of functions in a specific polynomial form
is still challenging.

For a Boolean function $f$ defined over $\GF{n}$ with $n = 2 m$ and
given in polynomial form, a classical characterization for bentness
is that its Walsh transform $\Wa{f}$ values are only $2^{\pm m}$.
Nevertheless, such a characterization is neither concise nor efficient:
the best algorithm to compute the full Walsh spectrum has complexity $O(n 2^n)$,
which is asymptotically optimal.
Whence the need to restrict to functions
in a given form and to look for more efficient criteria.
Unfortunately, only a few infinite families of Boolean functions
with a simple and efficient criterion for bentness are known.

The most classical family is due to Dillon~\cite{MR2624542}
and is made of monomial functions:
\[
f_a(x) = \tr{n}{a x^{r(2^m-1)}} \enspace ,
\]
where $n = 2 m$, $a \in \GF{n}^*$ and $r$ is co-prime with $2^m + 1$.
Such functions are bent (and even hyper-bent: for any $r$ coprime
with $(2^n - 1)$ the function $f_a(x^r)$ is also bent)
if and only if the Kloosterman sum $K_m(a)$
associated with $a$ is equal to zero~\cite{MR2624542,DBLP:journals/tit/Leander06,DBLP:journals/tit/CharpinG08}.
Not only does such a criterion gives a concise and elegant characterization for bentness,
but using the connection between Kloosterman sums and elliptic curves~\cite{MR925289,MR1054286}
it also allows to check for bentness in polynomial time~\cite{DBLP:conf/seta/Lisonek08,DBLP:journals/corr/abs-1104-3882}.
Further results on Kloosterman sums involving $p$-adic arithmetic~\cite{MR2794931,6126036,Moloney:PHD}
lead to even faster generation of zeros of Kloosterman sums and so of (hyper-)bent functions.

Mesnager~\cite{DBLP:journals/dcc/Mesnager11,DBLP:journals/tit/Mesnager11} proved a similar criterion
for a family Boolean functions in binomial form:
\[
f_{a,b}(x) = \tr{n}{a x^{r(2^m-1)}} + \tr{2}{b x^{\frac{2^n-1}{3}}} \enspace ,
\]
where $n = 2 m$, $a \in \GF{n}^*$, $b \in \GF[4]{}^*$
and $r$ is co-prime with $2^m + 1$
(but also $r = 3$ which divides $2^m+1$~\cite{DBLP:conf/ima/Mesnager09}).
When the extension degree $n$ is twice an odd number, that is when $m$ is odd,
$f_{a,b}$ is (hyper-)bent if and only if $K_m(a) = 4$.
Moreover, (hyper-)bent functions in this family can be quickly generated
as techniques used to generate zeros of Kloosterman sums can be transposed
to the value $4$~\cite{DBLP:conf/ima/FloriMC11}.

Unfortunately, the proof of the aforementioned characterization
does not extend to the case where $m$ is even.
Nevertheless, it is easy to show that $K_m(a) = 4$ is still a necessary
condition for $f_{a,b}$ to be bent in this latter case
(but note that $f_{a,b}$ can no longer be hyper-bent).
Further experimental evidence gathered by Flori, Mesnager
and Cohen~\cite{DBLP:conf/ima/FloriMC11} supported the conjecture
that it should also be a sufficient condition:
for $m$ up to $16$, $f_{a,b}$ is bent if and only if $K_m(a) = 4$.

In this note, the polar decomposition of fields of even extension degree
$n = 2^\nu m$ with $m$ odd is used to reduce the evaluation of the Walsh transform
of $f_{a,b}$ at $\omega \in \GF{n}^*$ to that of a Gauss sum of the form
\begin{align}
\label{eq:gengauss}
\sum_{u \in U} \mulch[n]{b \tr[m]{n}{\omega u}} \addch{\tr{n}{a u^{2^{2^{\nu-1}m}-1}}} \enspace ,
\end{align}
where $\GF{n}^*$ is decomposed as $\GF{n}^* \simeq U \times \GF{m}^*$,
$\mulch[n]{}$ is a cubic multiplicative character
and $\addch{}$ a quadratic additive character.

In the case of extensions of degree four times an odd number,
that is when $n$ is four times an odd number $m$,
an explicit formula involving the Kloosterman sum $K_{n/2}(a)$ is proved
for $\omega$ lying in the subfield $\GF{n/2}$,
and conjectured and supported by extensive experimental evidence
when $\omega \in \GF{n}$.
In particular, the validity of this formula would prove the following
conjecture for extensions of degree four times an odd number
(and give hope to prove the conjecture for $n$ of any $2$-adic valuation):
\begin{conjecture}
\label{cnj:kloofour}
Let $n = 4m$ with $m$ odd, $a \in \GF{n/2}^*$ and $b \in \GF[4]{}^*$.
The function $f_{a,b}$ is bent if and only if $K_{n/2}(a) = 4$.
\end{conjecture}

\section{Notation}
\label{sec:notation}

\subsection{Field trace}

\begin{definition}[Field trace]
For extension degrees $m$ and $n$ such that $m$ divides $n$,
the field trace from $\GF{n}$ down to $\GF{m}$ is denoted by $\tr[m]{n}{x}$.
\end{definition}

\subsection{Polar decomposition}

\begin{definition}[Extension degrees]
Let $n \geq 2$ be an even integer and
$\nu \geq 1$ denote its $2$-adic valuation.
We denote by $m_i$ for $0 \leq i \leq \nu$ the integer $n / 2^i$,
\eg $m_0 = n$ and $m_\nu = m$ in the introduction.
\end{definition}

For $0 \leq i < \nu$, the multiplicative group  $\GF{m_i}^*$
can be split using the so-called polar decomposition
\begin{align*}
\GF{m_i}^* & \simeq U_{i+1} \times \GF{m_{i+1}}^* \enspace ,
\end{align*}
where $U_{i+1} \subset \GF{m_i}^*$ is the subgroup of $(2^{m_{i+1}}+1)$-th roots of unity
and $\GF{m_{i+1}}^*$ the subgroup of $(2^{m_{i+1}}-1)$-th roots of unity.
Repeating this construction yields the following decomposition.

\begin{lemma}[Polar decomposition]
Let $\nu \geq 1$ and denote by $U$ denote the image of
$U_1 \times \cdots \times U_\nu$ within $\GF{m_0}^*$.
Then $\GF{m_0}^*$ decomposes as
\begin{align*}
\GF{m_0}^*
& \simeq U_1 \times \cdots \times U_\nu \times \GF{m_\nu}^*
\\
& \simeq U \times \GF{m_\nu}^*
\enspace .
\end{align*}
\end{lemma}

\subsection{Hilbert's Theorem 90}
\label{sec:ht90}

\begin{definition}
For $1 \leq i \leq \nu$ and $j \in \GF{}$, let $\T_{m_i}^j$ be the set
\begin{align*}
\T_{m_i}^j & = \set{x \in \GF{m_i}, \tr{m_i}{x^{-1}} = j}
\end{align*}
of elements of $\GF{m_i}$ whose inverses have trace $j$
(defining $0^{-1}$ to be $0$).
\end{definition}

Hilbert's Theorem 90~\cite{DBLP:journals/ffa/DillonD04}
implies that the function $x \mapsto x + x^{-1}$ is
$2$-to-$1$ from $U_i \backslash \set{1}$ to $\T_{m_i}^1$
and from $\GF{m_i}^* \backslash \set{1}$ to $\T_{m_i}^0 \backslash \set{0}$
(and both $0$ and $1$ are sent onto $0$).

\subsection{Dickson polynomials}
\label{sec:dickson}

\begin{definition}
We denote by $D_3$ the third Dickson polynomial of the first kind
$D_3(x) = x^3 + x$.
\end{definition}

A notable property of $D_3$ is that $D_3(x + x^{-1}) = x^3 + x^{-3}$.
It implies in particular that $D_3$ induces a permutation of $\T_{m_1}^0$
when $m_1$ is odd and of $\T_{m_1}^1$ when $m_1$ is
even~\cite[Propositions~5, 6 and Theorem~7]{DBLP:journals/ffa/DillonD04}.

\subsection{Characters}
\label{sec:characters}

\begin{definition}[Additive character]
Denote by $\addch{}$ the non-principal quadratic additive character of $\GF{}$.
\end{definition}

Together with the field trace,
$\addch{}$
can be used to construct all quadratic additive
characters of $\GF{m_i}$ for any $0 \leq i \leq \nu$.

\begin{definition}[Multiplicative character]
The non-principal cubic multiplicative character $\mulch[m_i]{}$ of $\GF{m_i}$
for any $0 \leq i < \nu$ is defined for $x \in \GF{m_i}$ as
\[
\mulch[m_i]{x} = x^{\frac{2^{m_i}-1}{3}} \enspace .
\]
\end{definition}

Note that if $x$ lies in a subextension,
that is $x \in \GF{m_{i+j}}$ with $0 \leq i+j < \nu$, then
\[
\mulch[m_i]{x} = \mulch[m_{i+j}]{x}^{2^j} \enspace .
\]

Remark that $3$ divides $2^{m_\nu}+1$ and is coprime with $2^{m_\nu}-1$ and $2^{m_i}+1$ for $0 \leq i < \nu$.
Therefore the function $x \mapsto x^3$ is a permutation of $\GF{m_\nu}^*$ and
$U_i$ for $1 \leq i < \nu$, and $3$-to-$1$ on $U_\nu$.
In particular, the multiplicative character $\mulch[m_0]{}$
is trivial everywhere on $\GF{m_0}^*$ but on $U_\nu$.

\subsection{Walsh transform}

\begin{definition}
The Walsh transform of a Boolean function $f$ at $\omega \in \GF{m_0}$ is
\begin{align*}
\Wa{f}(\omega) & = \sum_{x \in \GF{m_0}} \addch{f(x) + \tr{m_0}{\omega x}} \enspace .
\end{align*}
\end{definition}

It is well-known that a Boolean function $f$ is bent if and only if
its Walsh transform only takes the values $2^{\pm m_1}$.

\subsection{Kloosterman sums}
\label{sec:kloo}

\begin{definition}
For $a \in \GF{m_1}$, the Kloosterman sum $K_{m_1}(a)$ is
\begin{align*}
K_{m_1}(a) & = \sum_{x \in \GF{m_1}} \addch{\tr{m_1}{a x + x^{-1}}} \enspace .
\end{align*}
\end{definition}

The following identities (proved using the map from Section~\ref{sec:ht90})
are well-known:
\begin{align*}
\sum_{u_1 \in U_1} \addch{\tr{m_0}{a u_1}}
& = 1 + 2 \sum_{t \in \T_{m_1}^1} \addch{\tr{m_1}{a t}} \\
& = 1 - 2 \sum_{t \in \T_{m_1}^0} \addch{\tr{m_1}{a t}} \\
& = 1 - K_{m_1}(a) \enspace .
\end{align*}

\subsection{Cubic sums}
\label{sec:cubic}

\begin{definition}
For $a, b \in \GF{m_1}$, the cubic sum $C_{m_1}(a, b)$ is
\begin{align*}
C_{m_1}(a, b) & = \sum_{x \in \GF{m_1}} \addch{\tr{m_1}{a x^3 + b x}} \enspace .
\end{align*}
\end{definition}

The possible values of $C_{m_1}(a, b)$ were determined by Carlitz~\cite{MR544577}
together with simple criteria involving $a$ and $b$.

The most important consequence of Carlitz's results in our context
is that $C_{m_1}(a, a) = \sum_{x \in \GF{m_1}} \addch{\tr{m_1}{a D_3(x)}} = 0$ if and only if
\begin{itemize}
\item $\tr{m_1}{\alpha} = 0$ for $\alpha \in \GF{m_1}^*$ such that $a = \alpha^3$
when $m_1$ is odd (in that case $a$ is always a cube),
\item and when there exists $\alpha \in \GF{m_1}^*$ such that $a = \alpha^3$
(that is $a$ is a cube or equivalently $\mulch{a} = 1$)
and $\tr[2]{m_1}{\alpha} \neq 0$ (that is the cube root's half-trace is non zero)
when $m_1$ is even.
\end{itemize}
Charpin \etal later deduced that both in the odd case~\cite{DBLP:journals/jct/CharpinHZ07}
and in the even case~\cite{4595463,DBLP:journals/dm/CharpinHZ09}
these conditions are equivalent to $K_{m_1}(a) \equiv 1 \pmod{3}$.

For completeness, the other possible values for $C_{m_1}(a, a)$ when $m_1$ is even follow:
\begin{itemize}
\item When $a$ is a cube and $\tr[2]{m_1}{\alpha} = 0$,
then $C_{m_1}(a, a) = 2^{m_2 + 1} \addch{\tr{m_1}{u_0^3}}$,
where $u_0$ is any solution to $u^4 + u = \alpha^4$,
that is $u_0 = \sum_{i=0}^{(m_2-3)/2} \alpha^{4^{2*i+2}} + \gamma$
for any $\gamma \in \GF[4]{}$.
\item When $a$ is not a cube,
then $C_{m_1}(a, a) = - 2^{m_2} \addch{\tr{m_1}{a u_0^3}}$,
where $u_0$ is the unique solution to $u^4 + u / a = 1$,
that is $u_0 = \mulch{a} \sum_{i = 0}^{m_2-1} a^{4^i} a^{(4^i - 1)/3}$.
\end{itemize}

Finally, Carlitz also proved the following result on $C_{m_1}(a, 0)$ when $m_1 = 2 m_2$ is even:
\begin{align*}
C_{m_1}(a, 0)
& = \left\{
\begin{array}{ll}
(-1)^{m_2+1} 2^{m_2+1} & \text{if $\mulch{a} = 1$,} \\
(-1)^{m_2} 2^{m_2} & \text{if $\mulch{a} \neq 1$.} \\
\end{array}
\right.
\end{align*}

\subsection{Binomial functions}

The binomial Boolean functions $f_{a,b}$ studied in this note are defined over $\GF{m_0}$.
\begin{definition}
For $\nu \geq 1$, $a \in \GF{m_0}^*$ and $b \in \GF[4]{}^*$,
we denote by $f_{a,b}$ the binomial function
\begin{align}
f_{a,b}(x)
& = \tr{m_0}{a x^{2^{m_1}-1}} + \tr{2}{b \mulch[m_0]{x}} \enspace .
\end{align}
We also define $f_a = f_{a,0}$ (corresponding to Dillon's monomial) and
$g_b(x) = \tr{2}{b \mulch[m_0]{x}}$.
\end{definition}

\section{Preliminaries}

\subsection{Field of definition of the coefficients}

First notice that it is enough to know how to evaluate the Walsh transform of
functions $f_{a,b}$ for $a \in \GF{m_1}^*$.
\begin{lemma}
Let $a \in \GF{m_0}^*$ be written as $a = \alpha \tilde{a}$
with $\alpha \in U_1$ and $\tilde{a} \in \GF{m_1}^*$
using the polar decomposition of $\GF{m_0}^*$.
Let $\tilde{\alpha} \in U_1$ be a square root of $\alpha$
and $\beta \in \GF[4]{}^*$ be $\beta = \mulch[m_0]{\alpha}^{-1}$.
Then
\begin{align*}
\Wa{f_{a,b}}(\omega) & = \Wa{f_{\tilde{a},\beta b}}(\tilde{\alpha} \omega) \enspace .
\end{align*}
\end{lemma}
\begin{proof}
Indeed, $x \mapsto \tilde{\alpha} x$ induces a permutation of $\GF{m_0}$,
$\tilde{\alpha}^{2^{m_1}-1} = \tilde{\alpha}^{-2} = \alpha^{-1}$,
and $\mulch[m_0]{\tilde{\alpha}} = \mulch[m_0]{\alpha}^{-1}$,
so that
\begin{align*}
\Wa{f_{a,b}}(\omega) & = \sum_{x \in \GF{m_0}} \addch{f_{a,b}(x) + \tr{m_0}{\omega x}} \\
& = \sum_{x \in \GF{m_0}} \addch{f_{a,b}(\tilde{\alpha} x) + \tr{m_0}{\omega \tilde{\alpha} x}} \\
& = \sum_{x \in \GF{m_0}} \addch{f_{\tilde{a},\beta b}(x) + \tr{m_0}{\omega \tilde{\alpha} x}} \\
& = \Wa{f_{\tilde{a},\beta b}}(\tilde{\alpha} \omega) \enspace . \qedhere
\end{align*}
\end{proof}
From now on we can suppose that $a \in \GF{m_1}^*$ without loss of generality.

\subsection{Polar decomposition}

The polar decomposition yields the following expression
for $f_{a,b}$.
\begin{lemma}
For $\nu \geq 1$, $a \in \GF{m_0}^*$ and $b \in \GF[4]{}^*$,
and $x \in \GF{m_0}^*$,
$f_{a,b}(x)$ is
\begin{align}
f_{a,b}(x) & = f_{a,b}(u) = f_a(u_1) + g_b(u_\nu) \enspace .
\end{align}
\end{lemma}

\begin{proof}
Notice that $f_{a,b}(x) = f_a(x) + g_b(x)$.
Moreover $f_a$ is trivial on $\GF{m_1}^*$
and $g_b$ is trivial everywhere but on $U_\nu$
as noted in Section~\ref{sec:characters}.
\end{proof}

We now split the sum expressing the Walsh transform of $f_{a,b}$
at $\omega \in \GF{m_0}$ using the polar decomposition of $\GF{m_0}^*$
as $\GF{m_0}^* \simeq U \times \GF{m_\nu}^*$.
We write $x \in \GF{m_0}^*$ as $x = u y$
for $u \in U$, and $y \in \GF{m_\nu}^*$.

\begin{lemma}
For $\nu \geq 1$, $a \in \GF{m_1}^*$ and $b \in \GF[4]{}^*$,
the Walsh transform of $f_{a,b}$ at $\omega \in \GF{m_0}$ is,
for $\omega = 0$:
\begin{align}
\Wa{f_{a,b}}(0) & = 1 + \left( 2^{m_\nu} - 1 \right) \sum_{u \in U} \addch{f_{a,b}(u)} \enspace , \label{eq:walshzero}
\end{align}
and for $\omega \neq 0$:
\begin{align}
\Wa{f_{a,b}}(\omega)
& = 1 - \sum_{u \in U} \addch{f_{a,b}(u)} + 2^{m_\nu} \sum_{u \in U, \tr[m_\nu]{m_0}{\omega u} = 0} \addch{f_{a,b}(u)} \enspace . \label{eq:walshunit}
\end{align}
\end{lemma}

\begin{proof}
Using the polar decomposition, the Walsh transform of $f_{a,b}$ at $\omega \in \GF{m_0}$ can indeed be written
\begin{align*}
\Wa{f_{a,b}}(\omega) & = \sum_{x \in \GF{m_0}} \addch{f_{a,b}(x) + \tr{m_0}{\omega x}} \\
& = 1 + \sum_{x \in \GF{m_0}^*} \addch{f_{a,b}(x) + \tr{m_0}{\omega x}} \\
& = 1 + \sum_{(u, y) \in U \times \GF{m_\nu}^*} \addch{f_{a,b}(u y)} \addch{\tr{m_0}{\omega u y}} \enspace .
\end{align*}
Note that $3$ divides $2^{m_\nu}+1$ so that $\frac{2^{m_0}-1}{3} = (2^{m_\nu}-1) \frac{2^{m_\nu}+1}{3} \prod_{i=1}^{\nu-1}(2^{m_i}+1)$ and $f_{a,b}(u y) = f_{a,b}(u)$.
Therefore
\begin{align*}
\Wa{f_{a,b}}(\omega) & = 1 + \sum_{u \in U} \addch{f_{a,b}(u)} \sum_{y \in \GF{m_\nu}^*} \addch{\tr{m_\nu}{\tr[m_\nu]{m_0}{\omega u} y}} \enspace .
\end{align*}
The sum ranging over $\GF{m_\nu}^*$ is equal to $-1$ when $\tr[m_\nu]{m_0}{\omega u} \neq 0$ and $2^{m_1}-1$ when $\tr[m_\nu]{m_0}{\omega u} = 0$.
In particular, when $\omega = 0$,  the trace is $0$ for all $u \in U$.
\end{proof}

To go further, the cases $\nu = 1$ and $\nu > 1$ have to be dealt with separately.

\section{Odd case}
\label{sec:odd}

In this section, it is supposed that $\nu = 1$, \ie $m_1$ is odd and $U = U_1$,
which is the case that Mesnager settled~\cite{DBLP:journals/dcc/Mesnager11,DBLP:journals/tit/Mesnager11} with the following proposition.
We recall the main ingredients and results of her work as
similar ideas will be used for the even case.

\begin{proposition}[\cite{DBLP:journals/dcc/Mesnager11,DBLP:journals/tit/Mesnager11}]
For $\nu = 1$, $a \in \GF{m_1}^*$ and $b \in \GF[4]{}^*$,
the Walsh transform of $f_{a,b}$ at $\omega \in \GF{m_0}$ is,
for $\omega = 0$:
\begin{align}
\Wa{f_{a,b}}(0)
= \left\{
\begin{array}{ll}
1 + \frac{2^{m_1}-1}{3} \left( 1 - K_{m_1}(a) - 4 C_{m_1}(a, a) \right) & \text{if $b = 1$,} \\
1 + \frac{2^{m_1}-1}{3} \left( 1 - K_{m_1}(a) + 2 C_{m_1}(a, a) \right) & \text{if $b \neq 1$,}
\end{array}
\right.
\end{align}
and for $\omega \neq 0$:
\begin{align}
\Wa{f_{a,b}}(\omega)
& = \left\{
\begin{array}{ll}
1 + 2^{m_1} \addch{f_{a,b}(w_1^{-1})} + \frac{1}{3} \left( 1 - K_{m_1}(a) - 4 C_{m_1}(a, a) \right) & \text{if $b = 1$,}\\
1 + 2^{m_1} \addch{f_{a,b}(w_1^{-1})} + \frac{1}{3} \left( 1 - K_{m_1}(a) + 2 C_{m_1}(a, a) \right) & \text{if $b \neq 1$.}
\end{array}
\right.
\end{align}
\end{proposition}

\begin{proof}
For $\omega \neq 0$, $\tr[m_1]{m_0}{\omega u_1} = 0$ if and only if $u_1 = w_1^{-1}$, so that
\begin{align*}
\sum_{u_1 \in U_1, \tr[m_1]{m_0}{\omega u_1} = 0} \addch{f_{a,b}(u_1)}
& = \addch{f_{a,b}(w_1^{-1})} \enspace .
\end{align*}

The only difficulty lies in the computation of $\sum_{u_1 \in U_1} \addch{f_{a,b}(u_1)}$ which can be done by splitting the sum on $U_1$ according to the value of $\mulch{u_1}$:
\begin{align*}
\sum_{u_1 \in U_1} \addch{f_{a,b}(u_1)}
& = \sum_{u_1 \in U_1} \addch{f_{a}(u_1)} \addch{g_b(u_1)} \\
& = \sum_{u_1\in U_1, b \mulch[m_0]{u_1} = 1} \addch{f_{a}(u_1)}
 - \sum_{u_1\in U_1, b \mulch[m_0]{u_1} \neq 1} \addch{f_{a}(u_1)} \\
& = 2 \sum_{u_1\in U_1, b \mulch[m_0]{u_1} = 1} \addch{f_{a}(u_1)}
 - \sum_{u_1\in U_1} \addch{f_{a}(u_1)} \enspace .
\end{align*}

As noted in Section~\ref{sec:kloo} the second sum is
\begin{align*}
\sum_{u_1 \in U_1} \addch{f_{a}(u_1)} & = 1 - K_{m_1}(a) \enspace .
\end{align*}

As far as the first one is concerned, let us denote it $\Snu[1]$.
As $m_1$ is odd,
using properties of the Dickson polynomial $D_3$ given in Section~\ref{sec:dickson},
one can show that for $b = 1$:
\begin{align*}
\Snu[1]
& = \frac{1}{3} \left( 1 - K_{m_1}(a) + 2 C_{m_1}(a, a) \right) \enspace .
\end{align*}
As $\Snu[1]$ takes the same value for both $b \neq 1$,
one deduces that for $b \neq 1$:
\begin{align*}
\Snu[1]
& = \frac{1}{3} \left( 1 - K_{m_1}(a) - C_{m_1}(a, a) \right) \enspace .
\qedhere
\end{align*}
\end{proof}

Results of Carlitz~\cite{MR544577} on $C_{m_1}(a, a)$ when $m_1$ is odd
yield a concise and easy to compute the Walsh transform of $f_{a,b}$
at any $\omega \in \GF{m_0}$.

Together with Charpin \etal results~\cite{4595463,DBLP:journals/dm/CharpinHZ09}
and the Hasse--Weil bound on $K_{m_1}(a)$, these formulae prove that
$f_{a,b}$ is (hyper-)bent if and only if $K_{m_1}(a) = 4$
as was noted by Mesnager~\cite{DBLP:journals/dcc/Mesnager11,DBLP:journals/tit/Mesnager11}.
\begin{theorem}[\cite{DBLP:journals/dcc/Mesnager11,DBLP:journals/tit/Mesnager11}]
For $\nu = 1$, $a \in \GF{m}^*$ and $b \in \GF[4]{}^*$, the function $f_{a,b}$ is bent if and only if $K_{m_1}(a) = 4$.
\end{theorem}

\section{Even case}

\subsection{General extension degree}

In this section, it is supposed that $\nu > 1$, \ie both $m_0$ and $m_1$ are even.
The main difference with the case $\nu = 1$ is that $3$ does now divide $2^{m_1}-1$ (in fact $2^{m_\nu}+1$) rather than $2^{m_1}+1$,
 and $\mulch[m_0]{u}$ does not depend on the value of $u_1$ (but only on that of $u_\nu$).

In particular, the computation of $\sum_{u \in U} f_{a,b}(u)$ becomes straightforward.
\begin{lemma}
For $\nu > 1$, $a \in \GF{m_1}^*$ and $b \in \GF[4]{}^*$,
\begin{align}
\sum_{u \in U} \addch{f_{a,b}(u)}
& = - \frac{2^{2^{\nu-1} m_\nu} - 1}{3 \left(2^{m_\nu} - 1\right)} \left( 1 - K_{m_1}(a) \right) \enspace . \label{eq:sumfab}
\end{align}
\end{lemma}

\begin{proof}
Splitting $U$ as $U \simeq U_1 \times \cdots \times U_\nu$, the sum can be rewritten:
\begin{align*}
\sum_{u \in U} \addch{f_{a,b}(u)}
& = \prod_{j=2}^{\nu-1} \left( 2^{m_j}+1 \right) \sum_{u_1 \in U_1} \addch{f_a(u_1)} \sum_{u_\nu \in U_\nu} \addch{g_b(u_\nu)}
\\
& = \prod_{j=2}^{\nu-1} \left( 2^{m_j}+1 \right) \frac{2^{m_\nu} + 1}{3} \left( 1 - K_{m_1}(a) \right) \sum_{c \in \GF[4]{}^*} \addch{\tr{2}{b c}}
\\
& = - \prod_{j=2}^{\nu-1} \left( 2^{m_j}+1 \right) \frac{2^{m_\nu} + 1}{3} \left( 1 - K_{m_1}(a) \right)
\enspace .
\end{align*}
Finally, using the identity $\left( 2^{2^j m_\nu} + 1 \right) \left( 2^{2^j m_\nu} - 1 \right) = \left( 2^{2^{j+1} m_\nu} - 1 \right)$, the product of the $\left(2^{m_j}+1 \right)$'s is
\begin{align*}
\prod_{j=2}^{\nu} \left( 2^{m_j}+1 \right)
& = \prod_{j=2}^{\nu} \left( 2^{2^{\nu-j} m_\nu}+1 \right)
= \frac{2^{2^{\nu-1} m_\nu} - 1}{2^{m_\nu} - 1}
\enspace .
\qedhere
\end{align*}
\end{proof}

The value of the Walsh transform at $\omega = 0$ given
by Equation~(\ref{eq:walshzero}) can now be simplified.
\begin{lemma}
For $\nu > 1$, $a \in \GF{m_1}^*$ and $b \in \GF[4]{}^*$,
the Walsh transform of $f_{a,b}$ at $\omega = 0$ is
\begin{align}
\Wa{f_{a,b}}(0)
& = 1 - \frac{2^{m_1} - 1}{3} \left( 1 - K_{m_1}(a) \right) \enspace .
\end{align}
\end{lemma}

As noted by Mesnager~\cite{DBLP:journals/dcc/Mesnager11,DBLP:journals/tit/Mesnager11},
the Hasse--Weil bound on $K_{m_1}(a)$ implies that,
if $f_{a,b}$ is bent, then $\Wa{f_{a,b}}(0) = 2^{m_1}$
and $K_{m_1}(a) = 4$.
\begin{proposition}[\cite{DBLP:journals/dcc/Mesnager11,DBLP:journals/tit/Mesnager11}]
For $\nu > 1$, $a \in \GF{m_1}^*$ and $b \in \GF[4]{}^*$, if the function $f_{a,b}$ is bent, then $K_{m_1}(a) = 4$.
\end{proposition}

Finally, the value of the Walsh transform at $\omega \neq 0$ given by Equation~(\ref{eq:walshunit}) is simplified as follows.
\begin{lemma}
For $\nu > 1$, $a \in \GF{m_1}^*$ and $b \in \GF[4]{}^*$,
the Walsh transform of $f_{a,b}$ at $\omega \in \GF{m_0}^*$ is
\begin{align}
\Wa{f_{a,b}}(\omega)
& = 1 + \frac{2^{2^{\nu-1} m_\nu} - 1}{3 \left(2^{m_\nu} - 1\right)} \left( 1 - K_{m_1}(a) \right)
+ 2^{m_\nu} \sum_{u \in U, \tr[m_\nu]{m_0}{\omega u} = 0} \addch{f_{a,b}(u)}
\enspace . \label{eq:sumu}
\end{align}
\end{lemma}

\subsection{Descending to an odd degree extension}

To simplify further Equation~(\ref{eq:sumu}),
the sum over $u \in U$ can be split into smaller sums
according to the extension $\GF{m_i}$
(with $1 \leq i \leq \nu$) where $\tr[m_i]{m_0}{u \omega}$ becomes $0$,
giving the following expression.

\begin{proposition}
\label{prp:snu}
For $\nu > 1$, $a \in \GF{m_1}^*$ and $b \in \GF[4]{}^*$,
and $\omega \in \GF{m_0}^*$, denote by $\Snu$ the sum
\begin{align}
S_\nu(a, b, \omega) & = \sum_{\tr[m_{\nu-1}]{m_0}{u \omega} \neq 0, \tr[m_\nu]{m_0}{u \omega} = 0, b \mulch[m_0]{u_\nu} = 1} \addch{f_a(u_1)} \label{eq:gauss} \enspace.
\end{align}
The Walsh transform of $f_{a,b}$ at $\omega \neq 0$ is
\begin{align}
\Wa{f_{a,b}}(\omega)
& = 1 - \frac{2 \cdot 2^{\left(2^{\nu-1}-1\right)m_\nu} - 1}{3} \left(1 - K_{m_1}(a)\right) \nonumber \\
& \qquad - \frac{2 \cdot 2^{\left( 2^{\nu-1} - 1 \right) m_\nu} \left( 2^{m_\nu - 1} - 1 \right)}{3} \addch{f_a(w_1)} \nonumber \\
& \qquad + 2^{m_{\nu} + 1} \Snu \enspace . \label{eq:walshfull}
\end{align}
\end{proposition}

\begin{proof}
The sum over $U$
can be divided into subsums $\sigma_i$ over $U_i$:
$\sum_{u \in U, \tr[m_\nu]{m_0}{\omega u} = 0} \addch{f_{a,b}(u)} = \sum_{i = 1}^\nu \sigma_i$ with
\begin{align*}
\sigma_i
& =  \sum_{\substack{\tr[m_{i-1}]{m_0}{u_1 \cdots u_{i-1} w_1 \cdots w_{i-1}} \neq 0, \\ \tr[m_i]{m_0}{u_1 \cdots u_i w_1 \cdots w_i} = 0, \\u_{i+1} \in U_{i+1}, \ldots, u_\nu \in U_\nu}}
\hspace{-4em} \addch{f_a(u_1)} \addch{g_b(u_\nu)}
\enspace .
\end{align*}

The first sum $\sigma_1$
can be simplified as Equation~(\ref{eq:sumfab}):
\begin{align}
\sigma_1
& = \prod_{j=2}^{\nu - 1} \left( 2^{m_j} + 1 \right) \addch{f_a(w_1^{-1})} \sum_{u_\nu \in U_\nu} \addch{g_b(u_\nu)}
\nonumber \\
& = - \prod_{j=2}^{\nu - 1} \left( 2^{m_j} + 1 \right) \frac{2^{m_\nu} + 1}{3} \addch{f_a(w_1^{-1})}
\nonumber \\
& = - \frac{2^{2^{\nu-1} m_\nu} - 1}{3\left(2^{m_\nu} - 1\right)} \addch{f_a(w_1^{-1})}
\enspace . \label{eq:sumfirst}
\end{align}

The last sum $\sigma_\nu$ can be split according to the value of
$\mulch[m_0]{u_\nu}$ as in Section~\ref{sec:odd}:
\begin{align}
\sigma_\nu
& = 2 \sum_{\substack{\tr[m_{\nu-1}]{m_0}{u \omega} \neq 0, \\ \tr[m_\nu]{m_0}{u \omega} = 0,\\ b \mulch[m_0]{u_\nu} = 1}} \addch{f_a(u_1)}
 - \sum_{\substack{\tr[m_{\nu-1}]{m_0}{u \omega} \neq 0, \\ \tr[m_\nu]{m_0}{u \omega} = 0}} \addch{f_a(u_1)} \enspace , \label{eq:sumlast}
\end{align}
where the first term is $2 S_\nu(a, b, \omega)$
and the second term is
\begin{align}
- \sum_{\substack{\tr[m_{\nu-1}]{m_0}{u \omega} \neq 0, \\ \tr[m_\nu]{m_0}{u \omega} = 0 }} \addch{f_a(u_1)}
& =
- \prod_{j=2}^{\nu-1} 2^{m_j} \sum_{u_1 \neq w_1^{-1}} \addch{f_a(u_1)}
\nonumber \\
& = - \prod_{j=2}^{\nu-1} 2^{m_j} \left( 1 - \addch{f_a(w_1^{-1})} - K_{m_1}(a) \right)
\nonumber \\
& = - 2^{2 \left(2^{\nu - 2} - 1\right) m_\nu} \left( 1 - \addch{f_a(w_1^{-1})} - K_{m_1}(a) \right)
\enspace ,
\label{eq:sumeasy}
\end{align}
as the product of the $2^{m_j}$'s is
\begin{align}
\prod_{j=2}^{\nu-1} 2^{m_j}
& = \prod_{j=2}^{\nu-1} 2^{2^{\nu-j} m_\nu}
= 2^{2^{\nu-2} \sum_{j=0}^{\nu-3} 2^{-j} m_\nu}
= 2^{2^{\nu-2} 2 \left( 1-2^{-\nu+2} \right) m_\nu}
\enspace .
\label{eq:prodpows}
\end{align}

For $\nu > 2$, the intermediate sums $\sigma_i$ for $2 < i < \nu$ are:
\begin{align*}
\sigma_i
& = \prod_{j=2}^{i-1} 2^{m_j} \prod_{j=i+1}^{\nu-1} \left( 2^{m_j}+1 \right) \sum_{u_1 \neq w_1^{-1}} \addch{f_a(u_1)} \sum_{u_\nu \in U_\nu} \addch{g_b(u_\nu)}
\\
& = - \prod_{j=2}^{i-1} 2^{m_j} \prod_{j=i+1}^{\nu-1} \left( 2^{m_j}+1 \right) \frac{2^{m_\nu}+1}{3} \left( 1 - \addch{f_a(w_1^{-1})} - K_{m_1}(a) \right)
\enspace .
\end{align*}

Fortunately, a simpler expression for the sum of
of the products of $2^{m_j}$'s and $(2^{m_j}+1)$'s
for $1 < j < \nu$ can be devised.
Indeed, for $k \geq 3$ and any rational number $m$,
the sum that we denote by $\Sigma(m, k)$ is
\begin{align}
\Sigma(m,k)
=
\sum_{i=2}^{k-1}
\left(
\prod_{j=2}^{i-1} 2^{2^{k-j} m}
\prod_{j=i+1}^{k} \left( 2^{2^{k-j} m}+1 \right)
\right)
& =
\frac{2^{2 \left( 2^{k-2} - 1 \right) m}-1}{2^{m} - 1}
\enspace .
\label{eq:sumprodpows}
\end{align}
The proof goes by induction on $k$.
For $k = 3$, the identity states
$2^{m}+1 = \frac{2^{2 m}-1}{2^{m}-1}$.
Let us now suppose that Equation (\ref{eq:sumprodpows}) is
verified up to some $k \geq 3$ for all rational numbers $m$'s.
The sum for $k+1$ is
\begin{align*}
\Sigma(m, k + 1)
&
=
\left( 2^m + 1 \right)
\Sigma(2 m, k)
+
\left( 2^{m} + 1 \right)
\prod_{j=2}^{k-1} 2^{2^{k-j} (2 m)}
\end{align*}
By induction and a variation of Equation (\ref{eq:prodpows}),
the identity is proved for $k+1$:
\begin{align*}
\Sigma(m, k + 1)
&
=
\left( 2^m + 1 \right)
\frac{2^{2 \left( 2^{k-2} - 1 \right) (2 m)}-1}{2^{2 m} - 1}
+
\left( 2^m + 1 \right)
2^{4 \left(2^{k-2}-1\right) m}
\\
&
=
\frac{2^{4 \left( 2^{k-2} - 1 \right) m}-1}{2^{m} - 1}
+
\frac{
\left( 2^{2 m} - 1 \right)
2^{4 \left(2^{k-2}-1\right) m}
}
{2^m - 1}
\\
& =
\frac{2^{2 \left( 2^{k-1} - 1 \right) m}-1}{2^{m} - 1}
\enspace .
\end{align*}

Setting $k = \nu$ and $m = m_\nu$ in Equation (\ref{eq:sumprodpows}) yields
\begin{align}
\sum_{i=2}^{\nu-1} \sigma_i
& =
- \frac{2^{2 \left( 2^{\nu-2} - 1 \right) m_{\nu}}-1}{3 \left( 2^{m_{\nu}}-1 \right)} \left( 1 - \addch{f_a(w_1^{-1})} - K_{m_1}(a) \right)
\enspace .
\label{eq:summiddle}
\end{align}

Note that for $\nu = 2$, both sides of the above equality
are zero.
Therefore, for any $\nu > 1$,
Equations~(\ref{eq:sumfirst}), (\ref{eq:sumlast}), (\ref{eq:sumeasy})
and  (\ref{eq:summiddle}),
lead to the following expression for the Walsh transform at $\omega \neq 0$:
\begin{align*}
\Wa{f_{a,b}}(\omega)
& = 1 + \frac{2^{2^{\nu-1} m_\nu} - 1}{3\left(2^{m_\nu} - 1\right)} \left( 1 - K_{m_1}(a) \right) \nonumber \\
& \qquad - 2^{m_\nu} \frac{2^{2^{\nu-1} m_\nu} - 1}{3\left(2^{m_\nu} - 1\right)} \addch{f_a(w_1^{-1})} \nonumber \\
& \qquad - 2^{m_\nu} \frac{2^{2 \left( 2^{\nu-2} - 1 \right) m_{\nu}}-1}{3 \left( 2^{m_{\nu}}-1 \right)} \left( 1 - \addch{f_a(w_1^{-1})} - K_{m_1}(a) \right) \nonumber \\
& \qquad - 2^{m_\nu} 2^{2 \left(2^{\nu - 2} - 1\right) m_\nu} \left( 1 - \addch{f_a(w_1^{-1})} - K_{m_1}(a) \right) \nonumber \\
& \qquad + 2^{m_\nu + 1} S_\nu(a, b, \omega)
\enspace ,
\end{align*}
which gives the announced expression by gathering independently
the terms in $\addch{f_a(w_1^{-1})}$ and $(1 - K_{m_1}(a))$.
\end{proof}

Unfortunately, making the remaining sum $S_\nu(a, b, \omega)$ explicit
is a hard problem.
Doing so is equivalent to evaluating a Gauss sum as in Equation~(\ref{eq:gengauss}):
an exponential sum involving a multiplicative character and an additive character.
In the next section, we manage to tackle the case $\nu = 2$
when $\omega \in \GF{m_1}^*$ (that is $w_1 = 1$)
and conjecture a partial formula when $\omega \not\in \GF{m_1}^*$.

\subsection{Four times an odd number}

From now on, it is supposed that $\nu = 2$, \ie $m_0$ is four times the odd number $m_2$.

For $\tr[m_2]{m_0}{u \omega}$ to be zero with $u_1 \neq w_1^{-1}$,
$u_2$ must be the polar part of
$\left( \omega_2 \tr[m_1]{m_0}{u_1 \omega_1} \right)^{-1}$
so that the sum of Equation~(\ref{eq:gauss}) becomes
\begin{align}
\label{eq:fourgauss}
\Snu[2]
& = \sum_{u_1 \neq w_1^{-1}, \mulch[m_0]{w_2 \tr[m_1]{m_0}{u_1 w_1}} = b} \addch{\tr{m_0}{a u_1^{-2}}} \enspace .
\end{align}

\subsubsection{The subfield case}
We now restrict to the case $w_1 = 1$, that is $\omega \in \GF{m_1}^*$ rather than $\omega \in \GF{m_0}^*$.

\begin{lemma}
For $a \in \GF{m_1}^*$ and $b \in \GF[4]{}^*$,
and $\omega \in \GF{m_1}^*$,
define $\gamma \in \GF[4]{}^*$ by $\gamma = b \mulch[m_1]{w_2}$.
Then
\begin{align}
\Snu[2]
& = 2 \sum_{t \in \T_{m_1}^1, \mulch[m_1]{t} = \gamma} \addch{\tr{m_1}{a t}}
\enspace .
\label{eq:fourgaussone}
\end{align}
\end{lemma}

\begin{proof}
As $w_1 = 1$, both the multiplicative and additive characters act
on the the same inputs so that
we can use the function $u_1 \mapsto u_1 + u_1^{-1}$
to transform the sum over $U_1$ of Equation~(\ref{eq:fourgauss})
into a sum over $\T_{m_1}^1$:
\begin{align*}
\Snu[2]
& =\sum_{u_1 \neq 1, \mulch[m_1]{u_1^2 + u_1^{-2}} = b \mulch[m_1]{w_2}} \addch{\tr{m_1}{a \left( u_1^{-2}+ u_1^{2} \right)}}
\\
& = \sum_{u_1 \neq 1, \mulch[m_1]{u_1 + u_1^{-1}} = b \mulch[m_1]{w_2}} \addch{\tr{m_1}{a \left( u_1 + u_1^{-1} \right)}}
\\
& = 2 \sum_{t \in \T_{m_1}^1, \mulch[m_1]{t} = b \mulch[m_1]{w_2}} \addch{\tr{m_1}{a t}}
\enspace .
\qedhere
\end{align*}
\end{proof}

Remark that the sum in Equation~(\ref{eq:fourgaussone}) can be seen
as a first step toward generalizing the sum computed in Section~\ref{sec:odd}
in the odd case:
rather than involving $u_1$ directly, it involves
its trace $t = \tr[m_1]{m_0}{u_1}$.

As is customary, the sum over $\T_{m_1}^1$ can be evaluated using sums
over all of $\GF{m_1}$:
\begin{align}
\Snu[2]
& = \sum_{x \in \GF{m_1}^*, \mulch[m_1]{x} = \gamma} \addch{\tr{m_1}{a x}}
- \sum_{x \in \GF{m_1}^*, \mulch[m_1]{x} = \gamma} \addch{\tr{m_1}{a x + x^{-1}}}
\enspace .
\end{align}
The first sum is easily seen to be a cubic sum
whereas the computation of the second sum is more involved.

\begin{proposition}
For $\nu = 2$, $a \in \GF{m_1}^*$ and $\gamma \in \GF[4]{}^*$.
Define $\alpha \in \GF[4]{}^*$ by $\alpha = \mulch[m_1]{a}$.
The following equality holds:
\begin{align}
\sum_{x \in \GF{m_1}^*, \mulch{x} = \gamma} \addch{\tr{m_1}{a x}}
& = \left\{
\begin{array}{ll}
\frac{2^{m_2+1} - 1}{3} & \text{if $\gamma = \alpha^{-1}$,} \\
\frac{-2^{m_2} - 1}{3} & \text{if $\gamma \neq \alpha^{-1}$.}
\end{array}
\right. \label{eq:twistedcubic}
\end{align}
\end{proposition}

\begin{proof}
Let $c \in \GF[2]{m_1}^*$ be such that $\mulch{c} = \gamma$.
We make the change of variables $x = cx$ to transform
the sum into a cubic sum:
\begin{align*}
\sum_{x \in \GF{m_1}^*, \mulch{x} = \gamma} \addch{\tr{m_1}{a x}}
& = \sum_{x \in \GF{m_1}^*, \mulch{x} = \mulch{c}} \addch{\tr{m_1}{a x}} \\
& = \sum_{x \in \GF{m_1}^*, \mulch{x} = 1} \addch{\tr{m_1}{a c x}} \\
& = \frac{1}{3} \sum_{x \in \GF{m_1}^*} \addch{\tr{m_1}{a c x^3}} \\
& = \frac{1}{3} \left(C_{m_1}(a c, 0) - 1\right) \enspace .
\end{align*}
Carlitz's results~\cite{MR544577} give explicit values for this cubic sum
when $m_1$ is even and $m_2$ is odd.
\end{proof}

\begin{proposition}
For $\nu = 2$, $a \in \GF{m_1}^*$ and $\gamma \in \GF[4]{}^*$.
Define $\alpha \in \GF[4]{}^*$ by $\alpha = \mulch[m_1]{a}$.
The following equality holds:
\begin{align}
\sum_{x \in \GF{m_1}^*, \mulch{x} = \gamma} \addch{\tr{m_1}{a x + x^{-1}}}
& = \left\{
\begin{array}{ll}
\frac{1}{3} \left(2 C_{m_1}(a, a) + K_{m_1}(a) - 1\right) & \text{if $\gamma = \alpha$,} \\
\frac{1}{3} \left(- C_{m_1}(a, a) + K_{m_1}(a) - 1\right) & \text{if $\gamma \neq \alpha$.}
\end{array}
\right. \label{eq:twistedkloo}
\end{align}
\end{proposition}

\begin{proof}
First remark that summing over the three possible values of $\gamma$ yields
\begin{align*}
\sum_{x \in \GF{m_1}^*} \addch{\tr{m_1}{a x + x^{-1}}}
& = K_{m_1}(a) - 1 \enspace .
\end{align*}
Moreover, making the change of variable $x = \left(a x\right)^{-1}$
shows that the sum takes the same value for $\gamma$ and
$\alpha^{-1} \gamma^{-1}$.
In particular, it takes the same value for $\alpha \beta$ and $\alpha \beta^2$,
where $\beta \in \GF[4]{}^*$ is a primitive third root of unity,
that is for the elements of $\GF[4]{}^*$ different from $\alpha$,
and this value can be deduced from the value for $\gamma = \alpha$
which we now compute.

Denote by $r$ a square root of $a$.
The change of variable $x = rx$ and properties of the Dickson polynomial $D_3$
when $m_1$ is even show that for $\gamma = \alpha = \mulch[m_1]{r^{-1}}$:
\begin{align*}
\sum_{x \in \GF{m_1}^*, \mulch{x} = \alpha} \addch{\tr{m_1}{a x + x^{-1}}}
& = \sum_{x \in \GF{m_1}^*, \mulch{x} = 1} \addch{\tr{m_1}{r \left( x + x^{-1} \right)}}
\\
& = \frac{1}{3} \sum_{x \in \GF{m_1}^*} \addch{\tr{m_1}{r \left( x^3 + x^{-3} \right)}} \nonumber \\
& = \frac{1}{3} \sum_{x \in \GF{m_1}^*} \addch{\tr{m_1}{r D_3(x + x^{-1})}}
\\
\\
& = \frac{1}{3} \left(2 \sum_{t \in \T_{m_1}^0} \addch{\tr{m_1}{r D_3(t)}} - 1\right)
\\
\\
& = \frac{1}{3} \left(2 C_{m_1}(r, r) - 2 \sum_{t \in \T_{m_1}^1} \addch{\tr{m_1}{r t}} - 1\right)
\\
& = \frac{1}{3} \left(2 C_{m_1}(r, r) + K_{m_1}(r) - 1\right)
\\
& = \frac{1}{3} \left(2 C_{m_1}(a, a) + K_{m_1}(a) - 1\right)
\enspace .
\qedhere
\end{align*}
%
\end{proof}

Equations~(\ref{eq:twistedcubic}) and (\ref{eq:twistedkloo}) give
the following expression for $\Snu[2]$.
\begin{theorem}
For $\nu = 2$, $a \in \GF{m_1}^*$ and $b \in \GF[4]{}^*$,
and $\omega \in \GF{m_1}^*$,
let $\gamma = b \mulch[m_1]{w_2}$.
Then the sum $\Snu[2]$ is
\begin{align}
\Snu[2]
& = \left\{
\begin{array}{ll}
\frac{1}{3} \left(2^{m_2 + 1} - 2 C_{m_1}(a, a) - K_{m_1}(a)\right)
& \text{if $\gamma = \alpha$ and $\gamma = \alpha^{-1}$,}\\
\frac{1}{3} \left(-2^{m_2} - 2 C_{m_1}(a, a) - K_{m_1}(a)\right)
& \text{if $\gamma = \alpha$ and $\gamma \neq \alpha^{-1}$,}\\
\frac{1}{3} \left(2^{m_2 + 1} + C_{m_1}(a, a) - K_{m_1}(a)\right)
& \text{if $\gamma \neq \alpha$ and $\gamma = \alpha^{-1}$,}\\
\frac{1}{3} \left(-2^{m_2} + C_{m_1}(a, a) - K_{m_1}(a)\right)
& \text{if $\gamma \neq \alpha$ and $\gamma \neq \alpha^{-1}$.}
\end{array}
\right.
\end{align}
\end{theorem}

Carlitz's results~\cite{MR544577} recalled in Section~\ref{sec:cubic}
can be used to make the cubic sum $C_{m_1}(a, a)$ explicit.
In the particular case where $K_{m_1}(a) \equiv 1 \pmod{3}$,
which is equivalent to $C_{m_1}(a, a) = 0$ and implies that $a$ is a cube,
the expression for $\Snu[2]$ gets very concise,
as does Equation~(\ref{eq:walshfull}) for the Walsh transform.
\begin{corollary}
\label{crl:walshsubfield}
For $\nu = 2$, $a \in \GF{m_1}^*$ with $K_{m_1}(a) \equiv 1 \pmod{3}$
and $b \in \GF[4]{}^*$, and $\omega \in \GF{m_1}^*$,
let $\gamma = b \mulch[m_1]{w_2}$.
Then the sum $\Snu[2]$ is
\begin{align}
\Snu[2]
& = \frac{2^{m_2+1} - K_{m_1}(a)}{3} - \tr{2}{\gamma} 2^{m_2}
\enspace .
\end{align}
and the Walsh transform at $\omega \neq 0$ is
\begin{align}
\Wa{f_{a,b}}(\omega)
& = \addch{\tr{2}{\gamma}} 2^{m_1} + \frac{4 - K_{m_1}(a)}{3} \enspace .
\end{align}
\end{corollary}
Note that Corollary~\ref{crl:walshsubfield} shows that
for $\omega \in \GF{m_1}^*$ the Walsh transform of $f_{a,b}$
at $\omega$
is that of a bent function if and only if $K_{m_1}(a) = 4$.

\subsection{A conjectural general formula}

The techniques used in the previous section do not apply to the general case
where $w_1 \neq 1$, \ie $\omega \in \GF{m_0}^*$.
The main reason being that the multiplicative and additive characters
of $\GF{m_1}$ act on different values, \eg
$r = \tr[m_1]{m_0}{w_1 u_1}$, or $s = \tr[m_1]{m_0}{w_1^{-1} u_1}$,
for one of them,
and $t = \tr[m_1]{m_0}{u_1}$
for the other one.
Considering $v = \tr[m_1]{m_0}{w_1}$, these values are related by $r + s = v t$.
Moreover, the sum $\Snu[2]$ takes the same value for $w_1$ and $w_1^{-1}$,
so there is hope to introduce enough symmetry to reduce the case $w_1 \neq 1$
to the case $w_1 = 1$.
Unfortunately, we could not devise a way to do so.

Yet, experimental evidence presented in more details in
Section~\ref{sec:expdata} suggests that the following conjecture,
which relates the value of $\Snu[2]$ for $w_1 = 1$ and $w_1 \neq 1$,
is true.
\begin{conjecture}
\label{cnj:walshconj}
For $\nu = 2$, $a \in \GF{m_1}^*$ with $K_{m_1}(a) \equiv 1 \pmod{3}$
and $b \in \GF[4]{}^*$, and $\omega \in \GF{m_0}^*$,
let $\gamma = b \mulch[m_1]{w_2}$.
There exists a Boolean function $\mystery$ such that
the sum $\Snu[2]$ is
\begin{align}
\Snu[2]
& = \frac{2^{m_2+1} - K_{m_1}(a)}{3}
- 2 f_a(w_1^{-1}) \frac{2^{m_2+1} - 1}{3}
- \mystery \addch{f_a(w_1^{-1})} 2^{m_2}
\enspace .
\end{align}
The Walsh transform at $\omega \neq 0$ is then
\begin{align}
\Wa{f_{a,b}}(\omega)
& = \addch{\mystery f_a(w_1^{-1})} 2^{m_1} + \frac{4 - K_{m_1}(a)}{3} \enspace .
\end{align}
\end{conjecture}
In particular, this conjecture implies Conjecture~\ref{cnj:kloofour}:
if $K_{m_1}(a) = 4$, then $f_{a,b}$ is bent.
(And Corollary~\ref{crl:walshsubfield} already does so
when $\omega \in \GF{m_1}^*$.)

\subsection{Experimental data}
\label{sec:expdata}

The computation of $\Snu[2]$ was implemented in C and assembly\footnote{%
The source code is available at
\url{https://github.com/jpflori/expsums}.},
using AVX extensions for the arithmetic of $\GF{m_0}^*$,
PARI/GP~\cite{PARI2} to compute the Kloosterman sums $K_{m_1}(a)$,
and Pthreads~\cite{6506091} for parallelization.

The computational cost of verifying Conjecture~\ref{cnj:walshconj}
can be somewhat leveraged using elementary properties of $\Snu[2]$:
\begin{itemize}
\item it only depends on the cyclotomic class of $a \in \GF{m_1}^*$,
\item it is the same for $w_1$ and $w_1^{-1}$,
\item the inner value can be computed at the same time for $u_1$ and $u_1^{-1}$.
\end{itemize}
Whatsoever, there are:
\begin{itemize}
\item
$3$ values of $\gamma \in \GF[4]{}^*$,
\item
$\tilde{O}(2^{m_1})$ values of $a \in \GF{m_1}^*$,
\item
$2^{m_1-1}$ values of $w_1 \in U_1 \backslash \set{1}$,
\item
$\tilde{O}(2^{m_1})$ operations in $\GF{m_0}$ for each triple $(\gamma, a, w_1)$.
\end{itemize}
Therefore, checking the conjectured formula for $\Snu[2]$ over $\GF{m_0}$
has time complexity $\tilde{O}(2^{3 m_1})$ which quickly becomes overcostly
(and is comparable to that of computing the Walsh spectrum
for every cyclotomic class of $a \in \GF{m_1}^*$ which has time complexity
$\tilde{O}(2^{3 m_1})$ as well but space complexity $\tilde{O}(2^{m_1})$).

Still, we checked Conjecture~\ref{cnj:walshconj}
\begin{itemize}
\item completely for $m_2 = 3, 5, 7, 9$,
\item for $i$ up to $3405$ where $a = z^i$
and $z$ is a primitive element of $\GF{m_1}$ for $m_2 = 11$.
\end{itemize}

Finally, assuming $K_{m_1}(a) \equiv 1 \pmod{3}$
and Conjecture~\ref{cnj:walshconj} is correct,
Parseval's equality yields the following relation:
\begin{align*}
\sum_{x \in \GF{m_0}^*} \addch{\mystery f_a(w_1^{-1})}
& = \frac{2^{m_1} - 1}{3} \left( K_{m_1}(a) - 1 \right) \\
& = \Wa{f_{a,b}}(0) - 1 \enspace .
\end{align*}
This is supported by experimental evidence that
there are exactly $2^{m_1 - 1} + (5/6) \left(K_{m_1}(a) - 4\right) + 3$
(respectively $2^{m_1 - 1} - (1/6)\left(K_{m_1}(a) - 4\right)$)
values of $w_1 \in U_1$
such that $\mystery f_a(w_1^{-1})$ is zero when $\gamma = 1$
(respectively $\gamma \neq 1$).

\section{Further research and open problems}

Hopefully, Conjecture~\ref{cnj:walshconj} can be proved using
similar techniques
as the ones used by Mesnager~\cite{DBLP:journals/dcc/Mesnager11,DBLP:journals/tit/Mesnager11}
and in this note.
Otherwise, more involved techniques could be tried, \eg considering a whole
family of sums as a whole and their geometric structure.
Another posibility would be to directly treat the general Gauss sums of
Equations~(\ref{eq:gengauss}) and~(\ref{eq:gauss}) without focussing
on the case $\nu = 2$.

\bibliographystyle{plain}
\bibliography{even}

\end{document}